\documentclass[letterpaper, 10 pt, conference]{ieeeconf}

\usepackage{graphicx}
\usepackage{subcaption}
\usepackage[english]{babel}
\usepackage{url}            % simple URL typesetting
\usepackage{booktabs}       % professional-quality tables
\usepackage{amsfonts}       % blackboard math symbols
\usepackage{nicefrac}       % compact symbols for 1/2, etc.
\usepackage{microtype}      % microtypography
\usepackage{lipsum}
\usepackage{amssymb}
\usepackage{amsmath}
\usepackage{mathptmx}
\usepackage{hyperref}       % hyperlinks
\usepackage{xcolor}
\usepackage{comment}
\usepackage{soul}

\usepackage{algorithm}
\usepackage[noend]{algpseudocode}
\makeatletter
\def\BState{\State\hskip-\ALG@thistlm}
\makeatother

\usepackage[english]{babel}

\newcommand{\ie}{\emph{i.e.}}

\newtheorem{theorem}{Theorem}
\newtheorem{lemma}{Lemma}

\newtheorem{definition}{Definition}

\usepackage{comment}

\begin{document}
\title{On the Impact of Gaslighting on Partially Observed Stochastic Control}

\author{Shutian~Liu
        and~Quanyan~Zhu
       % Shutian~Liu,~{\it Graduate Student Member,~IEEE}
       % and~Quanyan~Zhu,~{\it Senior Member,~IEEE}
       % <-this % stops a space
%\thanks{This research is partially supported }
\thanks{The authors are with the Department of Electrical and Computer Engineering, Tandon School of Engineering,
        New York University, Brooklyn, NY, 11201 USA (e-mails: sl6803@nyu.edu; qz494@nyu.edu).}\vspace{-6mm}       
}

\pagestyle{empty}

\maketitle
\thispagestyle{empty}

%\overrideIEEEmargins

\begin{abstract}
Recent years have witnessed a significant increase in cyber crimes and system failures caused by misinformation.
Many of these instances can be classified as gaslighting, which involves manipulating the perceptions of others through the use of information.
In this paper, we propose a dynamic game-theoretic framework built on a partially observed stochastic control system to study gaslighting.
The decision-maker (DM) in the game only accesses partial observations, and she determines the controls by constructing information states that capture her perceptions of the system.
The gaslighter in the game influences the system indirectly by designing the observations to manipulate the DM's perceptions and decisions.
We analyze the impact of the gaslighter's efforts using robustness analysis of the information states and optimal value to deviations in the observations.
A stealthiness constraint is introduced to restrict the power of the gaslighter and to help him stay undetected.
We consider approximate feedback Stackelberg equilibrium as the solution concept and estimate the cost of gaslighting.
 
\end{abstract}


% \fi
%\IEEEpeerreviewmaketitle

\section{Introduction}
\label{sec:intro}

Gaslighting is a form of socio-psychological abuse that manipulates people's beliefs and causes them to question themselves and their environment. It involves an epistemic battle between the gaslighter and the victim, where the gaslighter cannot directly influence the environment. This battle resembles the strategic management of observations and manipulation of perceptions among decision-makers (DMs). To successfully gaslight a DM, the gaslighter must stimulate variations in the DM's beliefs about the environment or perceptions of sensory information, which requires sophisticated control and design of information.

%As gaslighting affects not only beliefs and perceptions but also the victim's awareness of the gaslighter's existence, it is an interdisciplinary topic that with increasing socio-psychological impacts with information design and belief control techniques in decision theory.\cite{sweet2019sociology,ruiz2020cultural,stark2019gaslighting}.

The increasing prevalence of information technologies, artificial intelligence (AI), and social media has made gaslighting inexpensive and convenient to conduct.  With networked cyber systems, attackers can leverage their ability to gather information from multiple sources in the system to generate disinformation and distribute it to users, gradually altering their perceptions and operational actions.  Our society is surrounded by various types of misinformation, and without the ability to detect, distinguish, and dissolve it, we are vulnerable to being gaslit into taking actions on behalf of malicious intentions. Therefore, it is crucial to achieve a deeper understanding of how gaslighters design information to manipulate perceptions. %For example, the increased connectivity and integration of networked cyber systems have made gaslighting a possible choice for attackers. Users in the networked systems often only observe regional information that propagate inside the local communities. On the contrary, attackers are eager to gather information from multiple sources in the system and figure out the status of the entire system. They can leverage this advantage to generate disinformation and gradually distribute it to the users to make them believe that their communication links are safe and their neighbors are benign. Once their perceptions about the environment change, users may adjust their operational actions and decrease their defending efforts, which may reveal potential vulnerabilities. Another example is advertising in social media. A company selling surveillance cameras can use AI to generate fake photos of crimes such as thefts caught by cameras. By constantly posting these photos on fake accounts on social media and emphasizing the losses due to the thefts, people would doubt the safety of their neighborhood and, as a result, determine to purchase the surveillance cameras. Our society is surrounded by various types of misinformation. If we are not able to detect, distinguish, and dissolve it, we can be easily gaslighted to take actions on behalf of malicious intentions. Therefore, there is a need to achieve a deeper understanding of how gaslighters design information to gaslight perceptions.

In this paper, we propose a finite-stage dynamic Stackelberg game framework \cite{bacsar1998dynamic} to model the gaslighting procedure. The gaslighter moves first at each stage, followed by the DM, capturing the sequential nature of the game where the gaslighter takes the leading role while the DM acts passively as the follower, unaware of the gaslighter's presence. The dynamic game framework is built on a partially observed stochastic control system, where the gaslighter has access to both the state and observation processes, while the DM only has access to the observation process. The gaslighter's advantage in full state information allows him to shape the DM's perceptions by designing observation distributions, while the DM's control decisions depend on her information states, which are generated using a version of the Bayes rule. The gaslighter limits his design efforts to avoid detection through a stealthiness metric and an approximate feedback Stackelberg solution, which breaks the direct dependence of gaslighting efforts on states, concealing state information from the DM.

This work examines the impact of gaslighting by using a metric induced by the $L^1$ norm defined for information states, which captures the influences of gaslighting on the DM's perceptions. The robustness of the information states to the gaslighting effort is used to measure the impact of gaslighting. We show that the variation in perception given samples of observations is upper-bounded by a quantity that accumulates the effective power of the gaslighting effort at each stage, irrespective of the control. We also investigate the impact of gaslighting on the optimal value received by the DM, extending the robustness analysis.

To restrict the power of the gaslighter and ensure that they stay undetected, we propose a measure of stealthiness. This measure is defined in a stage-wise manner to guarantee the gaslighter's stealthiness at each time when the DM obtains a new observation and can judge its credibility. When the gaslighting effort is stealthy, the upper bound on the deviation of the DM's optimal value reduces to a compact form. These robustness results hold for generic partially observed stochastic control problems. We adopt the solution concept of approximate feedback Stackelberg equilibrium, where the gaslighter uses suboptimal gaslighting efforts, reducing the computation burden of finding exact optimal actions. This approximation adds an additional level of difficulty for the DM in inferring the true state information. We also provide estimations of the gaslighter's costs when their effort is stealthy, characterizing the anticipated consequence of the gaslighter when they aim to gaslight without getting detected.

We review the related works in Section \ref{sec:related}. 
The dynamic Stackelberg game framework for modeling gaslighting will be presented in Section \ref{sec:problem} after we introduce a partially observed stochastic control problem that serves as the foundation.
We present the analysis of the impact of gaslighting and the solution of the gaslighting game in Section \ref{sec:analysis}.
Section \ref{sec:conclusion} concludes the paper.

\section{Related Works}
\label{sec:related}
Our framework builds upon the classic information state approach to partially observed stochastic control problems \cite{elliott2017discrete,james1994risk,kumar2015stochastic}, but takes a game-theoretic perspective and focuses on the impact of information design. Unlike previous work, which either formulated the problem as a small noise limit or used the H-infinity criteria, we take a distributional viewpoint and consider the design of observation distributions. Additionally, the gaslighter in our framework takes a generic role and has their own objective function to optimize, resulting in a nonzero-sum setting.

Our framework is closely related to the literature on one-sided partially observable stochastic games (OSPOSGs) \cite{zheng2022stackelberg, horak2023solving,horner2010markov,li2022commitment}. OSPOSGs model various strategic relations in reality and avoid the unnecessary derivation of the belief hierarchy. The role of the gaslighter in our framework is different, as they cannot control the state of the game directly. Instead, they must influence the state evolution through the actions of the other player by constructing certain observation distributions. This indirectness makes it challenging to investigate the impact of gaslighting on the system. Inspired by \cite{mcdonald2022robustness, kara2022robustness}, we address this challenge by studying the robustness of information states and optimal value to deviations in observations, measured by a metric suggested in \cite{elliott2017discrete}.

Our work also contributes to the emerging literature on strategic perception manipulation. The Stackelberg game formulation has been adopted in \cite{liu2022stackelberg} to study the design of risk perceptions in a static setting. The design of information has been considered in \cite{liu2022eproach} to study reopening policies from quarantines and in \cite{paarporn2022strategically} to investigate the potential advantages of revealing privileged knowledge. In \cite{sasaki2023strategic}, the author has also considered strategic manipulation in group decisions.

\section{Problem Formulation}
\label{sec:problem}
\subsection{Preliminary}
\label{sec:preliminary}
Consider the following discrete-time system:
\begin{equation}
    \begin{cases}      x_{k+1}=b(x_k,u_k)+w_k, \\
    y_{k+1}=h(x_k)+v_k
    \end{cases}
    \label{eq:system}
\end{equation}
on a probability space $(\Omega,\mathcal{F},\mathcal{P}^u)$ over a finite time interval $k=0,1, \cdots,K.$
The process $x\in \mathcal{X}\subset \mathbb{R}^n$ is the state process of the system which is unobservable to the decision-maker (DM).
The process $y\in \mathcal{Y}\subset \mathbb{R} $ denotes the observations of the DM.
Assume that  $w_k\in \mathbb{R}^n$ is i.i.d. for all $k=0,1,\cdots,K$ with density function $\psi(w)$.
Similarly, assume that $v_k\in \mathbb{R}$ is i.i.d. for all $k=0,1,\cdots,K$ with density function $\phi(v)$, and is independent of $x_0$ and $w_k$ for all $k=0,1,\cdots,K$.
The control $u_k$ at stage $k$ lives in a compact set $U_k\subset \mathbb{R}^m$.

Define the probability measure $ \mathcal{P}^\dagger$ by introducing the following Radon-Nikodym derivative:
\begin{equation}
    \frac{d\mathcal{P}^u}{d\mathcal{P}^\dagger}=Z_k=\prod_{i=1}^k\Psi(x_{i-1},y_i),
    \label{eq:change of measure}
\end{equation}
where 
\begin{equation}
    \Psi(x,y):=\frac{\phi(y-h(x))}{\phi(y)}.
    \label{eq:Psi}
\end{equation}
Then, under measure $\mathcal{P}^\dagger$, the observations $y_1,\cdots,y_K$ are i.i.d. following density $\phi$.

The cost function of the DM admits the following form:
\begin{equation}
    J(u)=\mathbb{E}^u\left[ \text{exp}\mu\left( \sum_{i=0}^{K-1}L(x_i,u_i)+\Phi(x_K) \right) \right],
    \label{eq:cost}
\end{equation}
where $\mu>0$ denotes the risk-sensitivity.

Let $\sigma_k\in L^1(\mathbb{R}^n)$ denote the information state at stage $k$.
The update of the information states is governed by the bounded linear operator $\Sigma^*: L^{\infty *}(\mathbb{R}^n) \rightarrow  L^{\infty *}(\mathbb{R}^n)$ defined as:
\begin{equation}    
\begin{aligned}
    \Sigma^*(u,y)\sigma(z)=&\int_{\mathbb{R}^n}\psi(z-b(\xi,u))\\
    &\cdot \text{exp}\left( \mu L(\xi,u) \right) \Psi(\xi, y)\sigma(\xi)d\xi,
    \label{eq:information state update}
\end{aligned}
\end{equation}
leading to the following recursion starting from the initial probability density $\rho$ on $\mathcal{X}$:
\begin{equation}
    \begin{cases}      \sigma_k=\Sigma^*(u_{k-1},y_k)\sigma_{k-1}, \\
    \sigma_0=\rho.
    \end{cases}
    \label{eq:information state recursion}
\end{equation}
For notational simplicity, define $T(u,y,z,\xi)=\psi(z-b(\xi,u))\text{exp}\left( \mu L(\xi,u) \right) \phi(y-h(\xi))$.

The loss function (\ref{eq:cost}) can be expressed in the following under the measure $\mathcal{P}^\dagger$:
\begin{equation}
    J(u)=\mathbb{E}^\dagger\left[ \prod_{i=1}^K \Psi(x_{i-1},y_i)\cdot\text{exp}\mu\left( \sum_{i=0}^{K-1}L(x_i,u_i)+\Phi(x_K) \right) \right],
    \label{eq:cost under P+}
\end{equation}
where the expectation $\mathbb{E}^\dagger[\cdot]$ can be computed since, after applying the change of measure (\ref{eq:change of measure}), the random variables $y_1,\cdots,y_K$ are i.i.d. with density $\phi$.
It is shown in \cite{james1994risk} that (\ref{eq:cost under P+}) can be equivalently represented using the information states as follows:
\begin{equation}
    \mathcal{J}(u):=\mathbb{E}^\dagger \left[ \int_{\mathbb{R}^n}\sigma_K(z)\text{exp}\left(\mu\Phi(z)\right)dz \right].
    \label{eq:cost information states}
\end{equation}

\subsection{The Gaslighting Game}
\label{sec:game formulation}

In this section, we describe a Stackelberg game between a DM and a gaslighter in a networked cyber system. The DM takes dynamic actions based on her local knowledge of the system, while the gaslighter intends to manipulate the DM's perceptions to influence her actions. The gaslighter has access to global information of the system $x$ and $y$, whereas the DM only observes the observation process $y$. This asymmetry of knowledge is modeled by the dynamical system (\ref{eq:system}). In the game, the DM acts as the follower and determines the controls of system (\ref{eq:system}). Her objective is to minimize the cost function (\ref{eq:cost}). Since the DM lacks information about the true states of the system, she constructs information states using her observations, which represent her perceptions of the environment. The gaslighter, or the leader, cannot directly influence the system's evolution but can manipulate the DM's perceptions by changing the observations received by her. Consequently, he tries to influence the system through the controls chosen by the DM.

%In this section, we introduce a dynamic Stackelberg game between the DM and the gaslighter.  The DM can be considered as a user in a networked cyber system who takes dynamic actions based on her local knowledge of the system. The gaslighter, can be thought of as a malicious individual who intend to manipulate the users' perceptions so that they take certain actions.  The gaslighter is assumed to observe global information of the system. This asymmetry of knowledge is captured by the dynamical system (\ref{eq:system}), where the DM only observes the observation process $y$ while the gaslighter observes both processes $x$ and $y$. The DM acts as the follower in the game and determines the controls of system (\ref{eq:system}).  Her objective is to minimize the cost function (\ref{eq:cost}).The DM holds an assumption of the networked system, which is represented  Due to the lack of information of the true states, the DM constructs information states using observations. These information states capture the DM's perceptions of the status of the environment. The gaslighter, or the leader in the game, cannot directly influence the evolution the system. Instead, he tries to influence the system through the controls chosen by the DM. His gaslighting effort changes the observations received by the DM, hence manipulates her perceptions of the environment.

Let $\mathbb{P}$ denote the set of probability density functions on $\mathbb{R}$.
The effort of the gaslighter is captured by a sequence of probability density functions $\phi^\circ=(\phi_1^\circ,\cdots,\phi_K^\circ)$ of the observations, \ie, $\phi^\circ_k\in \mathbb{P}$ for all $k=1,2,\cdots,K$.
The effort $\phi_k^\circ(y_k)$ at stage $k$ has an impact on the generation of the information state of stage $k$, which represents the perspective of the DM in perceiving the state $x_k$.
In particular, under the gaslighting effort $\phi_1^\circ,\cdots,\phi_K^\circ$, the update of the information states $\sigma^\circ_1,\sigma^\circ_2,\cdots,\sigma^\circ_K$ become:
\begin{equation}  
\begin{aligned}  
 \sigma^\circ_{k+1}&=\Sigma^{*\circ}_k(u,y)\sigma^\circ_k(z)\\
 &\hspace{-5mm}=\int_{\mathbb{R}^n}\psi(z-b(\xi,u))\text{exp}\left( \mu L(\xi,u) \right) \Psi^\circ_k(\xi, y)\sigma^\circ_k(\xi)d\xi,
    \label{eq:information state update gaslighted}    
\end{aligned}
\end{equation}
where 
%\begin{equation}
  $  \Psi^\circ_k(x,y)=\frac{\phi(y-h(x))}{\phi^\circ_k(y)}.$
 %   \label{eq:Psi circ}
%\end{equation}
Note that %the term (\ref{eq:Psi circ})
it does not influence the measure $\mathcal{P}^\dagger$.
The interpretation of the above setting is that there is a mismatch between the DM's assumption and the reality of the environment. 
In the DM's assumption, the observations $y_1,y_2,\cdots,y_K$ are i.i.d. and follow the probability density $\phi$.
So, she uses $\phi$ to construct the probability measure $\mathcal{P}^\dagger$.
However, in reality, the gaslighter is present, and the DM's information states are generated according to (\ref{eq:information state update gaslighted}) under the gaslighting effort $\phi^\circ$.

The goal of the gaslighter  is to guide the system to reach a favorable state in finite stages.
We use the function $\Gamma:\mathcal{X}\rightarrow \mathbb{R}$ to denote the terminal state cost and the function $H:\mathbb{P}\rightarrow \mathbb{R}_+$ to represent the cost of design with $H(\phi)=0$.
The gaslighter's problem is summarized as follows:
\begin{equation}
    \min_{\phi^\circ} \  \ \mathcal{I}(\phi^\circ):= \mathbb{E}^{u^\circ}\left[ \text{exp}(\mu \Gamma(x_K))\right] -\gamma  + \sum_{i=1}^K H(\phi^\circ_i),
    \label{eq:cost of gaslighter}
\end{equation}
where $\mathbb{E}^{u^\circ}$ denotes the expectation operator with respect to the probability measure induced by the control $u^\circ=(u^\circ_1,u^\circ_2,\cdots,u^\circ_K)\in U_1\times U_2 \times \cdots \times U_K$ of the DM when her observations follow $\phi^\circ$ and $\gamma:=\mathbb{E}^u\left[ \text{exp}(\mu \Gamma(x_K) )\right]$ with $u$ representing the control of the DM when her observations are i.i.d. and follow $\phi$.
Note that the term $\gamma$ in (\ref{eq:cost of gaslighter}) is a normalizing constant since its value is independent of $\phi^\circ$. 

The DM's problem in reaction to the gaslightinig effort $\phi^\circ$ admits the following form:
\begin{equation}
    \min_{u^\circ} \  \  \mathcal{J}^\circ(u^\circ):=\mathbb{E}^\dagger \left[ \int_{\mathbb{R}^n}\sigma^\circ_K(z)\text{exp}\left(\mu\Phi(z)\right)dz \right],
    \label{eq:cost of DM}
\end{equation}
where the information states $\sigma^\circ_1,\sigma^\circ_2,\cdots,\sigma^\circ_K$ are generated according to (\ref{eq:information state update gaslighted}).

\section{Impact of Gaslighting}
\label{sec:analysis}
In this section, we first investigate the impact of gaslighting by analyzing the robustness of the DM's information states and optimal value to the gaslighting effort.
During the analysis, we introduce constraints on the gaslighting effort to guarantee stealthiness.
When the gaslighting is stealthy, the robustness results take simple forms.
Then, we leverage the robustness  to provide estimations of the gaslighter's optimal design cost, which centers around the solution concept of approximate feedback Stackelberg equilibrium. 

\subsection{Robustness to Gaslighting Effort}
\label{sec:analysis:robust}
The analysis follows two steps. 
First, we calibrate the robustness of the information state process with respect to changes in the distribution of observation process. 
Then, we leverage these results to study the robustness of the optimal value with respect to the deviations in the information state process.

Let $\sigma_0,\sigma_1,\cdots,\sigma_K$ denote the original information state process updated according to $\sigma_{k+1}=\Sigma^*(u_{k},y_{k+1})\sigma_k$ with the distribution of observations following $\phi$.
Let $\sigma_0^\circ,\sigma_1^\circ,\cdots,\sigma_K^\circ$ denote the gaslighted information state process updated according to $\sigma^\circ_{k+1}=\Sigma^{*\circ}_{k+1}(u_k,y_{k+1})\sigma^\circ_k$ with the distribution of observations following $\phi^\circ_{k+1}(y_{k+1})$.
We use the notation $\sigma_{k+1}^\phi(\sigma_k)$ to denote the information state at stage $k+1$ obtained by using an operator involving $\phi$ on the previous information state $\sigma_k$.

Recall that an information state $\sigma$ is an unnormalized density function, \ie, $\sigma\in L^1(\mathbb{R}^n)$.
A metric induced by the $L^1$ norm  can be defined for $\sigma^1, \sigma^2 \in L^1(\mathbb{R}^n)$ as
\begin{equation}
    d(\sigma^1,\sigma^2)=||\sigma^1-\sigma^2||_{L^1}=\int_{\mathbb{R}^n}|\sigma^1(z)-\sigma^2(z)|dz.
    \label{eq:L1 metric}
\end{equation}

\paragraph{Robustness of information states}
Let $\hat{\phi}:=\text{max}_y \phi(y)$ and $l:=\text{max}_{x,u}\text{exp}\left(\mu L(x,u)\right)$.
\begin{lemma}
\label{lemma:robustness fix phi}
Given $\Bar{\sigma}_k, \hat{\sigma}_k \in L^1(\mathbb{R}^n)$ and $\phi$, the following inequality holds for a given observation $y$ and all controls $u$:
\begin{equation}
    d(\sigma_{k+1}^\phi(\Bar{\sigma}_k),\sigma_{k+1}^\phi(\hat{\sigma}_k))
    \leq
    \phi^{-1}(y) \hat{\phi} l\cdot d(\Bar{\sigma}_k,\hat{\sigma}_k).
    \label{eq:robustness fix phi}
\end{equation}
\end{lemma}
\begin{proof}
From (\ref{eq:information state update}) and (\ref{eq:L1 metric}), we obtain the following equations:
\begin{equation}
    \begin{aligned}
        &\  \  d(\sigma_{k+1}^\phi(\Bar{\sigma}_k),\sigma_{k+1}^\phi(\hat{\sigma}_k)) \\
        =& \int_{\mathbb{R}^n}\left|[\sigma_{k+1}^\phi(\Bar{\sigma}_k)](z) -[\sigma_{k+1}^\phi(\hat{\sigma}_k)](z)  \right|dz \\
        =&\phi^{-1}(y)\int_{\mathbb{R}^n}\int_{\mathbb{R}^n}T(u,y,z,\xi)\left| \Bar{\sigma}_k(\xi)-\hat{\sigma}_k(\xi) \right|d\xi dz \\
        =& \phi^{-1} \int_{\mathbb{R}^n}\left( \int_{\mathbb{R}^n}T(u,y,z,\xi)dz\right)\cdot \left| \Bar{\sigma}_k(\xi)-\hat{\sigma}_k(\xi) \right| d\xi \\
        =& \phi^{-1} \int_{\mathbb{R}^n} \Tilde{T}(u,y,\xi)\cdot \left| \Bar{\sigma}_k(\xi)-\hat{\sigma}_k(\xi) \right| d\xi,
        \label{eq:ineq fix phi 1}
    \end{aligned}
\end{equation}
where 
\begin{equation*}
    \begin{aligned}
        \Tilde{T}(u,y,\xi):=&\int_{\mathbb{R}^n}T(u,y,z,\xi)dz\\
        =&\int_{\mathbb{R}^n}\psi(z-b(\xi,u))dz\cdot \text{exp}\left(\mu L(\xi,u)\right)\phi(y-h(\xi))\\
        =&\text{exp}\left(\mu L(\xi,u)\right)\phi(y-h(\xi)).
    \end{aligned}
\end{equation*}
Using H\"{o}lder's inequality, (\ref{eq:ineq fix phi 1}) leads to
\begin{equation}    
\begin{aligned}
\  \ &d(\sigma_{k+1}^\phi(\Bar{\sigma}_k),\sigma_{k+1}^\phi(\hat{\sigma}_k)) 
          \leq  \phi^{-1}(y) \cdot \\
          &\left( \int_{\mathbb{R}^n} |\Bar{\sigma}_k(\xi)-\hat{\sigma}_k(\xi)|^pd\xi \right)^{1/p}\cdot \left( \int_{\mathbb{R}^n} |\Tilde{T}(u,y,\xi)|^qd\xi \right)^{1/q},
          \label{eq:ineq fix phi 2}
\end{aligned}
\end{equation}
for $p,q\in [1,\infty]$ with $1/p+1/q=1$.
Choosing $p=1$ and $q=\infty$, (\ref{eq:ineq fix phi 2}) reduces to
\begin{equation*}
    \begin{aligned}
          d(\sigma_{k+1}^\phi(\Bar{\sigma}_k),\sigma_{k+1}^\phi(\hat{\sigma}_k)) 
       & \leq  
        \phi^{-1}(y)d(\Bar{\sigma}_k,\hat{\sigma}_k)\text{sup}_{\xi} \Tilde{T}(u,y,\xi) \\
       &\leq 
         \phi^{-1}(y) \hat{\phi} l \cdot d(\Bar{\sigma}_k,\hat{\sigma}_k).
    \end{aligned}
\end{equation*}
This completes the proof.
\end{proof}

Since $\sigma_k\in L^1(\mathbb{R}^n)$ for all $k=0,1,\cdots,K$, we use $\zeta>0$ to denote the maximum $L^1$-norm of all possible information states, \ie, $\zeta=\max_{k, \sigma_k}||\sigma_k||_{L^1}$.
\begin{lemma}
\label{lemma:robustness fix sigma}
Given  $\phi$, $\phi^\circ$, and $\sigma_k\in L^1(\mathbb{R}^n)$, the following inequality holds for a given observation $y$ and all controls $u$:
\begin{equation}
    d(\sigma_{k+1}^{\phi^\circ}(\sigma_k),\sigma_{k+1}^\phi(\sigma_k))\leq \hat{\phi}l\zeta\cdot \left| \frac{1}{\phi^\circ(y)}-\frac{1}{\phi(y)} \right|.
    \label{eq:robustness fix sigma}
\end{equation}
\end{lemma}
\begin{proof}
From (\ref{eq:information state update}) and (\ref{eq:L1 metric}), we obtain the following equations:
\begin{equation}
    \begin{aligned}
          &\  \  d(\sigma_{k+1}^{\phi^\circ}(\sigma_k),\sigma_{k+1}^\phi(\sigma_k)) \\
          =&   \int_{\mathbb{R}^n}\left|[\sigma_{k+1}^{\phi^\circ}(\sigma_k)](z) -[\sigma_{k+1}^\phi(\sigma_k)](z)  \right|dz \\
          =& \left| \frac{1}{\phi^\circ(y)}-\frac{1}{\phi(y)} \right|\cdot \int_{\mathbb{R}^n}\int_{\mathbb{R}^n}|T(u,y,z,\xi)\sigma_k(\xi)|d\xi dz \\
          =& \left| \frac{1}{\phi^\circ(y)}-\frac{1}{\phi(y)} \right|\cdot \int_{\mathbb{R}^n}\Tilde{T}(u,y,\xi)\sigma_k(\xi)d\xi.
          \label{eq:ineq fix sigma 1}
    \end{aligned}
\end{equation}
Since $\int_{\mathbb{R}^n}\Tilde{T}(u,y,\xi)\sigma_k(\xi)d\xi\leq +\infty$ for all $u$ and $y$, we obtain the following inequalit:
\begin{equation}
    \begin{aligned}
        &\  \ \int_{\mathbb{R}^n}\Tilde{T}(u,y,\xi)\sigma_k(\xi)d\xi \\
        =& \int_{\mathbb{R}^n} \text{exp}\left(\mu L(\xi,u)\right)\phi(y-h(\xi)) \sigma_k(\xi)d\xi \\
        \leq& \int_{\mathbb{R}^n} \max_u\text{exp}\left(\mu L(\xi,u)\right)\max_y\phi(y-h(\xi)) \sigma_k(\xi)d\xi \\
        \leq& \hat{\phi} l \zeta.
        \label{eq:ineq fix sigma 2}
    \end{aligned}
\end{equation}
Combining (\ref{eq:ineq fix sigma 1}) and (\ref{eq:ineq fix sigma 2}) leads to (\ref{eq:robustness fix sigma}).
\end{proof}

Let $c:=\hat{\phi} l$.
Let $d_k:=d(\sigma^\circ_k,\sigma_k)$.
Let $Y_k(y_k):=\left|\frac{1}{\phi^\circ_k(y_k)}-
\frac{1}{\phi(y_k)}\right|$ for $k=1,2,\cdots,K$.
The impact of gaslighting on the perceptions of the DM is summarized in the following result.
\begin{theorem}
\label{thm:robustness of info states}
Given observations $y_1,y_2,\cdots,y_k$, the deviation of information states at stage $k$ under the gaslighting efforts $\phi^\circ_1,\phi^\circ_2,\cdots,\phi^\circ_{k}$ satisfies the following inequality regardless of the controls $u_0,u_1,\cdots,u_{k-1}$:
\begin{equation}
\begin{aligned}
    d_k\leq & \frac{c^k d_0}{\phi(y_k)\cdot \cdots \cdot \phi(y_1)}
    +\frac{c^k\zeta Y_1(y_1)}{\phi(y_k)\cdot \cdots \cdot \phi(y_2)} \\
    &+\frac{c^{k-1}\zeta Y_2(y_2)}{\phi(y_k)\cdot \cdots \cdot \phi(y_3)}+\cdots 
    + \frac{c^2\zeta Y_{k-1}(y_{k-1})}{\phi(y_k)}
    +c\zeta Y_{k}(y_k).
    \label{eq:robustness of info states}
\end{aligned}    
\end{equation}
\end{theorem}
\begin{proof}
We first observe the following relation using the triangle inequality:
\begin{equation*}
    \begin{aligned}
        d_{k+1}=&d(\sigma_{k+1}^{\phi^\circ_{k+1}}(\sigma^\circ_k),\sigma_{k+1}^\phi(\sigma_k))\\
        \leq & d(\sigma_{k+1}^{\phi^\circ_{k+1}}(\sigma^\circ_k),\sigma_{k+1}^\phi(\sigma^\circ_k))
        +d(\sigma_{k+1}^{\phi}(\sigma^\circ_k),\sigma_{k+1}^\phi(\sigma_k)).
    \end{aligned}
\end{equation*}
Then, the assertions in Lemma \ref{lemma:robustness fix phi} and Lemma \ref{lemma:robustness fix sigma} lead to 
\begin{equation}
    d_{k+1}\leq \frac{cd_k}{\phi_{k+1}(y_{k+1})}+c\zeta \left|\frac{1}{\phi^\circ_{k+1}(y_{k+1})}-\frac{1}{\phi(y_{k+1})}\right|.
    \label{eq:one step relation of dk}
\end{equation}
An induction on (\ref{eq:one step relation of dk}) leads to (\ref{eq:robustness of info states}).
\end{proof}

\paragraph{Stealthiness of gaslighting effort}
We restrict the power of the gaslighter by proposing the following definition of stealthiness.

\begin{definition}
\label{def:ESS}[Expected stage-wise stealthiness (ESS)]
The gaslighting effort $\phi^\circ_k$ is $s$-ESS at stage $k=1,2,\cdots,K$, if the following inequality holds:
\begin{equation}
    \sup_{u_{k-1}\in U_{k-1},\sigma_{k-1}\in L^1(\mathbb{R}^n)} \mathbb{E}_{\phi(y_k)}\left[ d(\sigma_{k}^{\phi^\circ_k}(\sigma_{k-1}),\sigma_{k}^{\phi}(\sigma_{k-1})) \right]\leq s,
    \label{eq:stagewise stealthiness}
\end{equation}
where $s>0$ denotes the trust level of the DM.
\end{definition}
The above stealthiness notion fulfills the following two goals.
Firstly, we aim to obtain a condition under which the gaslighting effort is stealthy on average instead of stealthy given specific observations.
Secondly, since the update of perception occurs at each stage, stealthiness must last for all $K$ stages. 
Note that in (\ref{eq:stagewise stealthiness}), the expectation is taken with respect to distribution $\phi$ instead of $\phi^\circ_k$.
This setting means that the stealthiness is evaluated from the perspective of the DM rather than the perspective of the gaslighter.
After all, it is the DM who tries to detect the existence of the gaslighter.

The following lemma presents a sufficient condition for ESS. 
We will use this condition to simplify the results of the robustness of optimal value.

\begin{lemma}
The gaslighting efffort $\phi^\circ_k$ is $s$-ESS at stage $k$, \ie, (\ref{eq:stagewise stealthiness}) holds at stage $k$, if the following holds for $\Bar{s}=s/(c\zeta)$:
\begin{equation}
    \int_{\mathbb{R}}\left|\frac{\phi(y)}{\phi^\circ_k(y)}-1\right|dy\leq \Bar{s}.
    \label{eq:ESS sufficient condition}
\end{equation}
\end{lemma}
\begin{proof}
The assertion follows from the following inequality:
\begin{equation*}
\begin{aligned}
       &\  \  \sup_{\substack{u_{k-1}\in U_{k-1},\\  \sigma_{k-1}\in L^1(\mathbb{R}^n)}}\mathbb{E}_{\phi(y_k)}\left[ d(\sigma_{k}^{\phi^\circ_k}(\sigma_{k-1}),\sigma_{k}^{\phi}(\sigma_{k-1})) \right] \\
       =&  \sup_{\substack{u_{k-1}\in U_{k-1},\\  \sigma_{k-1}\in L^1(\mathbb{R}^n)}}\int_{\mathbb{R}}\phi(y)\left|\frac{1}{\phi^\circ_k(y)}-\frac{1}{\phi(y)}\right|\int_{\mathbb{R}^n}\Tilde{T}(u,y,\xi)\sigma_k(\xi)d\xi dy \\
       \leq& \int_{\mathbb{R}}\phi(y)\left|\frac{1}{\phi^\circ_k(y)}-\frac{1}{\phi(y)}\right|dy
       \\
       & \  \ \   \
       \cdot
       \sup_{\substack{u_{k-1}\in U_{k-1},\\  \sigma_{k-1}\in L^1(\mathbb{R}^n)}} 
       \int_{\mathbb{R}^n}\text{exp}\left(\mu L(\xi,u)\right)\hat{\phi} \sigma_k(\xi)d\xi \\
       \leq& c\zeta\int_{\mathbb{R}}\phi(y)\left|\frac{1}{\phi^\circ_k(y)}-\frac{1}{\phi(y)}\right|dy.
\end{aligned}
\end{equation*}
\end{proof}
While condition (\ref{eq:stagewise stealthiness}) is straightforward, the sufficient condition (\ref{eq:ESS sufficient condition})  directly connects ESS with the gaslighting effort $\phi^\circ$.
In particular, the left-hand side of (\ref{eq:ESS sufficient condition}) is a measure of the deviation between probability density functions $\phi$ and $\phi^\circ_k$.

Note that the stealthiness level in Definition \ref{def:ESS} can be stage-dependent. In the ensuing sections, we will assume identical stealthiness levels for all stages $k=1,2,\cdots,K$ for simplicity.
The analysis can be extended to the case of stage-dependent stealthiness level with slight complications in notations.

\paragraph{Robustness of optimal value}
Recall that the DM's objective function is 
\begin{equation*}
    \mathcal{J}_{\sigma_K}(u):=\mathbb{E}^\dagger \left[\int_{\mathbb{R}^n}\sigma_K(z)\text{exp}\left(\mu\Phi(z)\right)dz\right],
\end{equation*}
where we use the notation $\mathcal{J}_{\sigma_K}$ to emphasize the fact that the cost is associated with the information state process $\sigma_0,\sigma_1,\cdots,\sigma_K$.
We are interested in the robustness of the optimal objective value with respect to the deviation of the information state process, \ie, the value of $\mathcal{J}_{\sigma^\circ_K}-\mathcal{J}_{\sigma_K}$.

Let $e_{\Phi}=\max_{x\in \mathcal{X}} \text{exp}\left(\mu\Phi(x)\right)$.
Let $\Tilde{d}_k(y_1,\cdots,y_k)$ denote the right-hand side of (\ref{eq:robustness of info states}).
The impact of gaslighting on the optimal value of the DM's problem is summarized in the following result.
\begin{theorem}
 \label{thm:robustness of optimal value}   
 The deviation of the DM's objective value under the gaslighting effort $\phi^\circ_1,\cdots,\phi^\circ_K$ satisfies the following inequality regardless of the controls $u_0,u_1,\cdots,u_{k-1}$:
 \begin{equation}
     \mathcal{J}_{\sigma^\circ_K}-\mathcal{J}_{\sigma_K}
     \leq e_{\Phi} \cdot \mathbb{E}^\dagger\left[ \Tilde{d}_K(y_1,\cdots,y_K) \right].
     \label{eq:robustness of optimal value}
 \end{equation}
\end{theorem}

\begin{proof}
From the definition of the DM's objective function, we observe that
\begin{equation}
    \begin{aligned}
            \mathcal{J}_{\sigma^\circ_K}-\mathcal{J}_{\sigma_K} 
           =& \mathbb{E}^\dagger\left[ \int_{\mathbb{R}^n}(\sigma^\circ_K(z)-\sigma_K(z))\text{exp}\left(\mu\Phi(z)\right)dz \right] \\
           =&\int_{\mathbb{R}^n}
            \mathbb{E}^\dagger\left[ \sigma^\circ_K(z)-\sigma_K(z)\right]
           \text{exp}\left(\mu\Phi(z)\right)dz.
           \label{eq:value deviation 1}
    \end{aligned}
\end{equation}
Using H\"{o}lder's inequality, we obtain the following inequality from (\ref{eq:value deviation 1}):
\begin{equation}
    \begin{aligned}
           &  \mathcal{J}_{\sigma^\circ_K}-\mathcal{J}_{\sigma_K} \\  
           \leq& \int_{\mathbb{R}^n}\left| \int_{\mathbb{R}}\cdots \int_{\mathbb{R}}
           \Pi_{i=1}^{K} \phi(y_i) \left[ \sigma^\circ_K(z)-\sigma_K(z) \right] dy_1\cdots dy_K \right| dz \\
           \  \ &\cdot \max_{z} \text{exp}\left(\mu\Phi(z) \right) \\
           \leq& \int_{\mathbb{R}}\cdots \int_{\mathbb{R}}
           \Pi_{i=1}^{K} \phi(y_i) \left[ \int_{\mathbb{R}^n}|\sigma^\circ_K(z)-\sigma_K(z)|dz \right] dy_1\cdots dy_K \\
           \  \ &\cdot \max_{z} \text{exp}\left(\mu\Phi(z) \right)\\
           \leq& e_\Phi\cdot \mathbb{E}^\dagger\left[d_K\right].
           \label{eq:value deviation 2}
    \end{aligned}
\end{equation}
Combining (\ref{eq:robustness of info states}) with (\ref{eq:value deviation 2}), we arrive at  the theorem.
\end{proof}

The impact of stealthy gaslighting effort admits a compact representation as follows.
\begin{theorem}
\label{thm:robustness of optimal value reduced} 
Suppose that the gaslighting efforts $\phi^\circ_1,\cdots,\phi^\circ_K$ satisfy (\ref{eq:ESS sufficient condition}) for $k=1,2,\cdots,K$, then, (\ref{eq:robustness of optimal value}) reduces to
\begin{equation}
     \mathcal{J}_{\sigma^\circ_K}-\mathcal{J}_{\sigma_K}
     \leq e_\Phi (c^Kd_0+s\sum_{i=1}^{K-1}c^i).
     \label{eq:robustness of optimal value when ESS}
\end{equation}
\end{theorem}
\begin{proof}
Taking expectation on both sides of (\ref{eq:one step relation of dk}) with respect to $\phi(y_1), \cdots,\phi(y_K)$, we obtain 
\begin{equation}
    \begin{aligned}
        \mathbb{E}^\dagger\left[d_k\right]
        \leq& \mathbb{E}^\dagger\left[ \frac{cd_{k-1}}{\phi_{k}(y_{k})}+c\zeta \left|\frac{1}{\phi^\circ_{k}(y_{k})}-\frac{1}{\phi(y_{k})}\right| \right] \\
        \leq&  \mathbb{E}_{\phi(y_k)}\left[ \frac{c}{\phi(y_k)} \right]\cdot  \mathbb{E}^\dagger\left[ d_{k-1} \right] \\
       \  \ & +c\zeta  \mathbb{E}_{\phi(y_k)}\left[ \left|\frac{1}{\phi^\circ_{k}(y_{k})}-\frac{1}{\phi(y_{k})}\right| \right] \\
        \leq& c \cdot  \mathbb{E}^\dagger\left[ d_{k-1} \right] +s,
        \label{eq:one step relation of E[dk]}
    \end{aligned}
\end{equation}
where the last inequality follows from the fact that (\ref{eq:ESS sufficient condition}) holds at stage $k$.
Combining an induction on (\ref{eq:one step relation of E[dk]}) with (\ref{eq:value deviation 2}), we arrive at   the theorem.
\end{proof}
The upper bound in (\ref{eq:robustness of optimal value when ESS}) shows the maximal performance of gaslighting effort that is stealthy in the sense of (\ref{eq:ESS sufficient condition}) evaluated using the deviation in the DM's optimal value.
From the perspective of the DM, a simple way to mitigate the impact of gaslighting is to decrease the trust level $s$.

\subsection{Approximate Feedback Stackelberg Equilibrium.}
\label{sec:analysis:equilibrium}
In this section, we introduce the solution concept of feedback $\epsilon$-Stackelberg equilibrium based on a dynamic programming procedure.
The robustness results developed in Section \ref{sec:analysis:robust} is then utilized to estimate the cost of the gaslighter.

In \cite{james1994risk}, the authors have shown that the control of the system (\ref{eq:system}) under partial observations can be carried out with the dynamic programming equations by minimizing the equivalent cost function (\ref{eq:cost information states}) leveraging the information state recursion (\ref{eq:information state recursion}).
Let $Z(\sigma,k)$ denote the value function associated with (\ref{eq:cost information states}) and (\ref{eq:information state recursion}).
The dynamic programming equations admit the following form \cite{james1994risk}:
\begin{equation}
    \begin{aligned}
        \begin{cases}      
        Z(\sigma_k,k)=\inf_{u_k\in U_K}\mathbb{E}^\dagger\left[ Z(\Sigma^*(u_k,y_{k+1})\sigma_k, k+1) \right], \\
        Z(\sigma_K,K)= \int_{\mathbb{R}^n}\sigma_K(z)\text{exp}\left(\mu\Phi(z)\right)dz.
    \end{cases}
    \label{eq:DP eq original}
    \end{aligned}
\end{equation}

In the gaslighting game formulated in Section \ref{sec:game formulation}, the stage costs of the gaslighter and the DM depend on the actions of both players.
Hence, there is a need to enrich the formulas in (\ref{eq:DP eq original}) by taking into account the strategic relations in the gaslighting game.
To proceed, we introduce the following functions.
Suppose that the gaslighting effort $\phi^\circ=(\phi^\circ_1,\phi^\circ_2,\cdots,\phi^\circ_K)$ and the control effort $u^\circ=(u^\circ_1,u^\circ_2,\cdots,u^\circ_K)$ are fixed.
Let $V(\sigma^\circ_k,\phi^\circ_k, u^\circ_k, k)$ be associated with the DM's control problem defined recursively as follows:
\begin{equation}
    \begin{aligned}
        \begin{cases}      
        &V(\sigma^\circ_k, \phi^\circ_k, u^\circ_k, k)=\\
        &\  \ \  \ 
        \mathbb{E}^\dagger\left[ V(\Sigma_k^{*\circ}(u^\circ_k,y_{k+1})\sigma^\circ_k, \phi^\circ_{k+1},  u^\circ_{k+1},k+1) \right], \\
        &V(\sigma^\circ_K, \phi^\circ_K, u^\circ_K, K)= \int_{\mathbb{R}^n}\sigma^\circ_K(z)\text{exp}\left(\mu\Phi(z)\right)dz,
    \end{cases}
    \label{eq:DP eq DM}
    \end{aligned}
\end{equation}
where the information states $\sigma^\circ_1,\sigma^\circ_2,\cdots,\sigma^\circ_K$ are generated according to (\ref{eq:information state update gaslighted}).
Let $W(x_k, \phi^\circ_k, u^\circ_k, k)$ be associated with the gaslighter's design problem.
Since the gaslighter has full observations and his objective function is additive, $W(x_k, \phi^\circ_k, u^\circ_k,k)$ is defined recursively by the following recursion:
\begin{equation}
    \begin{aligned}
        \begin{cases}      
        &W(x_k, \phi^\circ_k, u^\circ_k, k)=H(\phi^\circ_k)\\
        &\  \ \   \ +\mathbb{E}^{u^\circ}\left[ W(b(x_k,u^\circ_k)+w_k, \phi^\circ_{k+1}, u^\circ_{k+1}, k+1) \right],
        \\
        &W(x_K, \phi^\circ_K,u^\circ_K,K)= \text{exp}(\mu \Gamma(x_K)).
    \end{cases}
    \label{eq:DP eq gaslighter}
    \end{aligned}
\end{equation}
The following result is a consequence of dynamic programming and the sequential nature of Stackelberg games. 

\begin{theorem}
\label{thm:stackelberg}
A pair of policies $(\Bar{\phi}^\circ, \Bar{u}^\circ)$ constitutes a feedback $\epsilon$-Stackelberg equilibrium for a given vector $\epsilon=(\epsilon_1,\epsilon_2,\cdots,\epsilon_K)$ with $\epsilon_k>0, \forall k=1,2,\cdots,K$, if the following conditions are satisfied:
\begin{equation}
    W(x_k,\Bar{\phi}^\circ_k, \Bar{u}^\circ_k, k)\leq \min_{\phi^\circ_k, u^\circ_k \in R_k(\phi^\circ_k)} W(x_k,\phi^\circ_k,u^\circ_K,k) + \epsilon_k, \forall x_k, \forall k ,
    \label{eq:condition epsilon Stackelberg leader}
\end{equation}
where $R_k(\phi^\circ_k)$ denotes the response set defined for all $\sigma^\circ_k$ and all $k$ as:
\begin{equation}
\begin{aligned}
    R_k(\phi^\circ_k)=&\{\Tilde{u}_k\in U_k: V(\sigma^\circ_k,\phi^\circ_k,\Tilde{u}_k,k) \\
    =&\min_{u_k\in U_k} V(\sigma^\circ_k,\phi^\circ_k,u_k,k) \}.
    \label{eq:reaction set}
\end{aligned}
\end{equation}
\end{theorem}

Approximate equilibrium enjoys at least the following two advantages.
Firstly, it relieves the computation burden.
Since the gaslighting effort is an infinite dimensional design object, approximation using step functions can be adopted in solving the equations in Theorem \ref{thm:stackelberg}.
Secondly, it prevents information leakage.
Since the gaslighter has full observation of the system, the DM can infer state information from the actions taken by the gaslighter.
However, by solving (\ref{eq:condition epsilon Stackelberg leader}), the gaslighter breaks the direct dependence of his action on the state information.
Since we consider a finite-stage dynamic game, it is challenging for the DM to accurately infer state information from the gaslighter's actions.
We refer to \cite{zheng2022stackelberg,laine2023computation} for more discussions on approximate equilibria.

The following result captures the consequence of stealthy gaslighting.
\begin{theorem}
\label{thm:estimation}
Suppose that (i) the design cost function is chosen as $H(\Tilde{\phi})=t \int_{\mathbb{R}}\left|\frac{\phi(y)}{\Tilde{\phi}(y)}-1\right|dy$ with $t>0$ denoting design cost parameter; 
(ii) the gaslighter's  effort $\phi^\circ$ satisfies (\ref{eq:ESS sufficient condition}) strictly for all $k=1,2,\cdots,K$.
Then, the cost of the gaslighter satisfies: 
\begin{equation}
    \mathcal{I}(\phi^\circ)\geq -e_\Gamma (\hat{\phi}^Kd_0+s\sum_{i=1}^{K-1}\hat{\phi}^i) + Kt\Bar{s},
    \label{eq:estimation stackelberg cost}
\end{equation}
where $e_{\Gamma}=\max_{x\in \mathcal{X}} \text{exp}\left(\mu\Gamma(x)\right).$
\end{theorem}
\begin{proof}
From Section \ref{sec:preliminary}, we know that the DM's objective function  admits two equivalent expressions, \ie, $J=\mathcal{J}$.
Then, we obtain the following inequality from Theorem \ref{thm:robustness of optimal value reduced}:
\begin{equation}
\begin{aligned}
    |J(u^\circ)-J(u)|=
    &\Bigg| \mathbb{E}^{u^\circ}\left[ \text{exp}\mu\left( \sum_{i=0}^{K-1}L(x_i,u^\circ_i)+\Phi(x_K) \right) \right] \\
    \  \ &-
    \mathbb{E}^u\left[ \text{exp}\mu\left( \sum_{i=0}^{K-1}L(x_i,u_i)+\Phi(x_K) \right) \right] \Bigg| \\
    \leq 
     & e_\Phi (c^Kd_0+s\sum_{i=1}^{K-1}c^i).
     \label{eq:proof:estimation}
\end{aligned}
\end{equation}
Choose $L(x,u)=0$ for all $x\in X$ and all $u\in U$ in (\ref{eq:proof:estimation}), we arrive at 
\begin{equation}
\begin{aligned}
    \left|\mathbb{E}^{u^\circ}\left[\text{exp}\mu \Phi(x_K)  \right]
    -
    \mathbb{E}^{u}\left[\text{exp}\mu \Phi(x_K)  \right] \right|
    \leq   e_\Phi (\hat{\phi}^Kd_0+s\sum_{i=1}^{K-1}\hat{\phi}^i).
    \label{eq:proof:estimation 1}
\end{aligned}
\end{equation}
Observe that the only term that depends on the function $\Phi$ on the right-hand side of (\ref{eq:proof:estimation 1}) is $e_\Phi$, we obtain the following inequality as a direct consequence of considering function $\Gamma$ in (\ref{eq:proof:estimation 1}):
\begin{equation}
    \left|\mathbb{E}^{u^\circ}\left[\text{exp}\mu \Gamma(x_K)  \right]
    -
    \mathbb{E}^{u}\left[\text{exp}\mu \Gamma(x_K)  \right] \right|
    \leq   e_\Gamma (\hat{\phi}^Kd_0+s\sum_{i=1}^{K-1}\hat{\phi}^i).
    \label{eq:proof:estimation 2}
\end{equation}
This leads to the assertion in the theorem.
\end{proof}
Note that an upper bound of $\mathcal{J}(\phi^\circ)$ can also be derived from inequality (\ref{eq:proof:estimation 2}).
However, the lower bound (\ref{eq:estimation stackelberg cost}) is more meaningful because of the following reason.
Since $\Gamma(\cdot)$ represents the terminal state cost of the gaslighter, it is reasonable that its value is lower when the final state is induced by the control $u^\circ$ than that induced by $u$.
Hence, $\mathbb{E}^{u^\circ}\left[ \text{exp}(\mu \Gamma(x_K))\right] \leq \gamma =  \mathbb{E}^{u}\left[\text{exp}(\mu \Gamma(x_K))  \right]$.
Then, the second term of the right-hand side of (\ref{eq:estimation stackelberg cost}) represents the maximum improvement of the gaslighter's terminal state cost when the gaslighting effort satisfies (\ref{eq:ESS sufficient condition}). 
The bound (\ref{eq:estimation stackelberg cost}) shows what the gaslighter can anticipate in the best-case scenario.

\section{Concluding Remarks}
\label{sec:conclusion}
This paper has proposed a dynamic Stackelberg game-theoretic framework to model gaslighting behavior and its impact. The framework builds on a partially observed stochastic control system, where information states capture the DM's perceptions. We extend would our framework to the setting of partially observed Markov decision processes, where state transitions are more generic. Further analysis would be on the sensitivity analysis of DM perceptions and optimal value concerning the gaslighting effort. By characterizing variations in control actions when observations change, sensitivity analysis offers a more detailed explanation of the impact of gaslighting than robustness results, where controls in worst-case scenarios are considered.

\bibliographystyle{IEEEtran}
\bibliography{bibliography.bib}

\begin{thebibliography}{10}
\providecommand{\url}[1]{#1}
\csname url@samestyle\endcsname
\providecommand{\newblock}{\relax}
\providecommand{\bibinfo}[2]{#2}
\providecommand{\BIBentrySTDinterwordspacing}{\spaceskip=0pt\relax}
\providecommand{\BIBentryALTinterwordstretchfactor}{4}
\providecommand{\BIBentryALTinterwordspacing}{\spaceskip=\fontdimen2\font plus
\BIBentryALTinterwordstretchfactor\fontdimen3\font minus
  \fontdimen4\font\relax}
\providecommand{\BIBforeignlanguage}[2]{{%
\expandafter\ifx\csname l@#1\endcsname\relax
\typeout{** WARNING: IEEEtran.bst: No hyphenation pattern has been}%
\typeout{** loaded for the language `#1'. Using the pattern for}%
\typeout{** the default language instead.}%
\else
\language=\csname l@#1\endcsname
\fi
#2}}
\providecommand{\BIBdecl}{\relax}
\BIBdecl

\bibitem{bacsar1998dynamic}
T.~Ba{\c{s}}ar and G.~J. Olsder, \emph{Dynamic noncooperative game
  theory}.\hskip 1em plus 0.5em minus 0.4em\relax SIAM, 1998.

\bibitem{elliott2017discrete}
R.~J. Elliott and J.~B. Moore, ``Discrete time partially observed control,'' in
  \emph{Differential Equations}.\hskip 1em plus 0.5em minus 0.4em\relax
  Routledge, 2017, pp. 481--490.

\bibitem{james1994risk}
M.~R. James, J.~S. Baras, and R.~J. Elliott, ``Risk-sensitive control and
  dynamic games for partially observed discrete-time nonlinear systems,''
  \emph{IEEE transactions on automatic control}, vol.~39, no.~4, pp. 780--792,
  1994.

\bibitem{kumar2015stochastic}
P.~R. Kumar and P.~Varaiya, \emph{Stochastic systems: Estimation,
  identification, and adaptive control}.\hskip 1em plus 0.5em minus 0.4em\relax
  SIAM, 2015.

\bibitem{zheng2022stackelberg}
W.~Zheng, T.~Jung, and H.~Lin, ``The stackelberg equilibrium for one-sided
  zero-sum partially observable stochastic games,'' \emph{Automatica}, vol.
  140, p. 110231, 2022.

\bibitem{horak2023solving}
K.~Hor{\'a}k, B.~Bo{\v{s}}ansk{\`y}, V.~Kova{\v{r}}{\'\i}k, and C.~Kiekintveld,
  ``Solving zero-sum one-sided partially observable stochastic games,''
  \emph{Artificial Intelligence}, vol. 316, p. 103838, 2023.

\bibitem{horner2010markov}
J.~H{\"o}rner, D.~Rosenberg, E.~Solan, and N.~Vieille, ``On a markov game with
  one-sided information,'' \emph{Operations research}, vol.~58, no. 4-part-2,
  pp. 1107--1115, 2010.

\bibitem{li2022commitment}
T.~Li and Q.~Zhu, ``Commitment with signaling under double-sided information
  asymmetry,'' \emph{arXiv preprint arXiv:2212.11446}, 2022.

\bibitem{mcdonald2022robustness}
C.~McDonald and S.~Yuksel, ``Robustness to incorrect priors and controlled
  filter stability in partially observed stochastic control,'' \emph{SIAM
  Journal on Control and Optimization}, vol.~60, no.~2, pp. 842--870, 2022.

\bibitem{kara2022robustness}
A.~D. Kara, M.~Raginsky, and S.~Y{\"u}ksel, ``Robustness to incorrect models
  and data-driven learning in average-cost optimal stochastic control,''
  \emph{Automatica}, vol. 139, p. 110179, 2022.

\bibitem{liu2022stackelberg}
S.~Liu and Q.~Zhu, ``Stackelberg risk preference design,'' \emph{arXiv preprint
  arXiv:2206.12938}, 2022.

\bibitem{liu2022eproach}
------, ``Eproach: A population vaccination game for strategic information
  design to enable responsible covid reopening,'' in \emph{2022 American
  Control Conference (ACC)}.\hskip 1em plus 0.5em minus 0.4em\relax IEEE, 2022,
  pp. 568--573.

\bibitem{paarporn2022strategically}
K.~Paarporn and P.~N. Brown, ``Strategically revealing capabilities in general
  lotto games,'' \emph{arXiv preprint arXiv:2211.14907}, 2022.

\bibitem{sasaki2023strategic}
Y.~Sasaki, ``Strategic manipulation in group decisions with pairwise
  comparisons: A game theoretical perspective,'' \emph{European Journal of
  Operational Research}, vol. 304, no.~3, pp. 1133--1139, 2023.

\bibitem{laine2023computation}
F.~Laine, D.~Fridovich-Keil, C.-Y. Chiu, and C.~Tomlin, ``The computation of
  approximate generalized feedback nash equilibria,'' \emph{SIAM Journal on
  Optimization}, vol.~33, no.~1, pp. 294--318, 2023.

\end{thebibliography}
\nocite{*}



%\newpage


%\vfill


\end{document}